\documentclass[12pt]{amsart}
\usepackage{amsfonts,amssymb, amscd, latexsym, graphicx, psfrag, color, float}
\usepackage[all]{xy}

\newtheorem{dummy}{dummy}[section]
\newtheorem{lemma}[dummy]{Lemma}
\newtheorem{theorem}[dummy]{Theorem}

\newtheorem{corollary}[dummy]{Corollary}
\newtheorem{proposition}[dummy]{Proposition}
\theoremstyle{definition}
\newtheorem{definition}[dummy]{Definition}

\newtheorem{remark}[dummy]{Remark}

\newcommand{\bC}{\mathbb{C}}

\newcommand{\bP}{\mathbb{P}}

\newcommand{\bR}{\mathbb{R}}
\newcommand{\bZ}{\mathbb{Z}}
\newcommand{\bT}{\mathbb{T}}

\newcommand{\cE}{\mathcal{E}}

\newcommand{\cL}{\mathcal{L}}

\newcommand{\cO}{\mathcal{O}}

\newcommand{\gk}{\kappa}

\newcommand{\Perf}{\mathcal{P}\mathrm{erf}}

\newcommand{\Tbi}{T_{\beta_{i}}}
\newcommand{\Tbj}{T_{\beta_{j}}}
\newcommand{\Tbk}{T_{\beta_{k}}}
\newcommand{\Tbkk}{T_{\beta_{k+1}}}
\newcommand{\Tbii}{T_{\beta_{i+1}}}
\newcommand{\Tbjj}{T_{\beta_{j+1}}}
\newcommand{\Ta}{T_{\alpha}}
\newcommand{\Txi}{T_{\kappa(x_i)}}
\newcommand{\Txj}{T_{\kappa(x_j)}}
\newcommand{\Txk}{T_{\kappa(x_k)}}
\newcommand{\To}{T_{\mathcal{O}}}
\newcommand{\Txii}{T_{\kappa(x_{i+1})}}
\newcommand{\Txjj}{T_{\kappa(x_{j+1})}}
\newcommand{\Txkk}{T_{\kappa(x_{k+1})}}
\newcommand{\Txu}{T_{\kappa(x_{1})}}
\begin{document}

\title[A note on mapping class group actions on derived categories.]{A note on mapping class group actions on derived categories.}

\begin{abstract}
Let $X_n$ be a cycle of $n$ projective lines, and $\bT_n$ a symplectic torus with $n$ punctures. Using the theory of spherical twists introduced by Seidel and Thomas \cite{ST}, I will define an action of the pure mapping class group of $\bT_n$ on $D^b(Coh(X_n))$. The motivation comes from homological mirror symmetry for degenerate elliptic curves, which was studied by the author with Treumann and Zaslow in \cite{STZ}.
\end{abstract}

\author{Nicol\`o Sibilla}
\address{Nicol\`o Sibilla, Max Planck Institute for Mathematics, Vivatsgasse 7,
53111 Bonn,
Germany}
\email{sibilla@mpim-bonn.mpg.de}

\maketitle

{\small \tableofcontents}

\section{Introduction.}

According to Kontsevich's \emph{Homological Mirror Symmetry} conjecture (from now on HMS, see \cite{K}), given a Calabi-Yau variety $X$ and a symplectic manifold $\tilde{X}$, if $X$ and $\tilde{X}$ are mirror partners, then the derived category of coherent sheaves over $X$, $D^b(Coh(X))$, should be equivalent to the Fukaya category of $\tilde{X}$, $Fuk(\tilde{X})$. Since $Fuk(\tilde{X})$ is an invariant of the symplectic geometry of $\tilde{X}$, mirror symmetry predicts that the group of symplectic automorphisms of $\tilde{X}$ acts by equivalences on $D^b(Coh(X))$. In \cite{ST} Seidel and Thomas investigate this aspect of HMS by introducing the notions of \emph{spherical object} and \emph{twist functor}, which can be defined for general triangulated categories, and axiomatize the formal homological properties enjoyed by equivalences of the Fukaya category induced by \emph{generalized Dehn twists} (these are special symplectic automorphisms introduced by Seidel, see \cite{S}). Using their theory they are able, in many interesting examples, to give a conjectural description of the equivalences of $D^b(Coh(X))$ which should be mirror to symplectic automorphisms of $\tilde{X}$. I refer the reader to \cite{ST} for a detailed account of this circle of ideas. A brief overview of the relevant definitions will be given in Section \ref{action} below.

Let $X_n$ be a cycle of $n$ projective lines, i.e. a nodal curve of arithmetic genus $1$, with $n$ singular points. Well known mirror symmetry heuristics suggest that the mirror of $X_n$ should be a symplectic torus with $n$ punctures, which I shall denote $\bT_n$. In the paper \cite{STZ}, joint with Treumann and Zaslow, we prove a version of HMS for $X_n$ and $\bT_n$, by showing that the category of perfect complexes over $X_n$, $\Perf(X_n)$, is quasi-equivalent to a certain conjectural model of $Fuk(\bT_n)$ which we develop in the paper. See also the recent work \cite{LP} in which the authors prove, with very different techniques, a HMS statement for the case $n = 1$. 

Motivated by \cite{STZ}, in this paper I explore the consequences of mirror symmetry for the study of auto-equivalences of $D^b(Coh(X_n))$. Recall that the mapping class group of an oriented surface $\Sigma$ can be described as the group of symplectic automorphisms of $\Sigma$, modulo isotopy. The existence of an action of the mapping class group of $\bT_n$ on $D^b(Coh(X_n))$ does not follow directly from \cite{STZ}, as the model of the Fukaya category considered there is not acted upon, in any obvious way, by symplectomorphisms of $\bT_n$.\footnote{Note also that the HMS statement in \cite{STZ} involves $\Perf(X_n)$, rather than the full derived category of $X_n$. However, $D^b(Coh(X_n))$ has the same group of auto-equivalences of $\Perf(X_n)$. In fact, any equivalence of $D^b(Coh(X_n))$ gives, by restriction, an equivalence of $\Perf(X_n)$, and it follows from Lemma \ref{equiv} that this assignment is a bijection.} My main result uses the framework of \cite{ST} to construct an action of the (pure) mapping class group of $\bT_n$, $\mathrm{PM}(\bT_n)$, over $D^b(Coh(X_n))$. In future work, I plan to establish that this action is faithful. It is worth pointing out that the action I will define is, in an appropriate sense, a categorification of the symplectic representation of the mapping class group, which can be recovered by considering the induced action on the \emph{numerical} Grothendieck group of $\Perf(X_n)$ (see Remark \ref{symplectic} below. For a definition of the symplectic representation, the reader can refer to \cite{FM}, Chapter 6). 

The paper is organized as follows. In Section \ref{mapping}, I give some background on the mapping class group, and then work out a convenient presentation of $\mathrm{PM}(\bT_n)$. The proof of the main result, Theorem \ref{main}, is contained in Section \ref{action}. Theorem \ref{main} generalizes previous results in \cite{ST} and in \cite{BK}, where the authors considered, respectively, the case of a smooth elliptic curve, and of the nodal cubic in $\bP^2$ (i.e. the case $n = 1$). Equivalences of $D^b(Coh(X_n))$ were also investigated in \cite{L}. However, as the author in \cite{L} restricts to a subgroup of equivalences satisfying certain homological conditions, which are violated by the spherical twists I shall consider below, there is essentially no overlap between his work and the present project.

\emph{Acknowledgments:} I am grateful to David Treumann and Eric Zaslow for many valuable conversations, and for our collaboration \cite{STZ}, which is the starting point of this paper. I thank Luis Paris for giving me very useful explanations concerning his paper with Catherine Labru\`ere \cite{LP}. I would like to thank Bernd Siebert and Hamburg University, and Yuri Manin and the Max Planck Institute for Mathematics, for their hospitality during a period in which part of this work was carried out.  

\section{The mapping class group of a punctured torus.}
\label{mapping}
In this section I will briefly review some basic facts about the mapping class group, and then give a presentation of the mapping class group of the punctured torus based on \cite{LP}. Also, it will be useful to spell out some relations between mapping classes which were found by Gervais in \cite{G}. For a comprehensive introduction to the mapping class group I refer the reader to \cite{FM}.

Let $\Sigma = \Sigma_{g, n, b}$ be a differentiable, oriented surface of genus $g$, with $n$ marked points, and $b$ boundary components. The mapping class group of $\Sigma$, denoted by $\mathrm{M}(\Sigma)$, is the group of isotopy classes of orientation preserving diffeomorphisms of $\Sigma$, which send the set of marked points to itself, and restrict to the identity on the boundary components. Note that $\mathrm{M}(\Sigma_{g, n, b})$ is uniquely determined by the parameters $g, n, b$.  The \emph{pure} mapping class group of $\Sigma$ is the subgroup $\mathrm{PM}(\Sigma) \hookrightarrow \mathrm{M}(\Sigma)$ of mapping classes fixing pointwise the set of marked points. Alternatively, $\mathrm{PM}(\Sigma)$ can be defined as the subgroup of $\mathrm{M}(\Sigma)$ generated by Dehn twists along simple closed curves (for the definiton of Dehn twist, and a proof of this claim, see Chapter 3 and 4 of \cite{FM}). In making the above definitions, marked points on $\Sigma$ could be interpreted, equivalently, as punctures, and I shall make use freely of both viewpoints in the following.

A surface $\Sigma$ with $n$ punctures and $b + m$ boundary components can be immersed in a surface with $n + m$ punctures and $b$ boundary components (we can trade $m$ boundary components for $m$ punctures, by gluing a punctured discs along each boundary component we wish to remove). Further, this immersion induces a map of pure mapping class groups. The details can be found in Section 2 of \cite{LP}, together with the following lemma which will be useful later.

\begin{lemma}
\label{punct-bd}
Let $(g, r, m) \notin \{(0,0,1),(0,0,2)\}$, then we have the exact sequence
$$
1 \rightarrow \bZ^{m} \rightarrow \mathrm{PM}(\Sigma_{g, n, b+m}) \rightarrow \mathrm{PM}(\Sigma_{g, n+m, b}) \rightarrow 1,
$$
where $\bZ^m$ stands for the free abelian group of rank $m$ generated by the Dehn twists along the $m$ boundary components we are removing.
\end{lemma}

Set $\bT_n = \Sigma_{1, n, 0}$ and $\bT_{n,m} = \Sigma_{1, n, m}$. The pure mapping class group $\mathrm{PM}(\bT_{n})$ is generated by Dehn twists along $n+1$ non-separating simple closed curves. In order to fix ideas, it is convenient to choose explicit representatives for this collection of curves. I will mostly follow the notation of \cite{G}, to which I refer for further details. Let $\Lambda = \bZ^2 \hookrightarrow \bR^2$ be the standard integral lattice, let $\bT = \bR^2 / \bZ^2$, and fix a fundamental domain for the action of $\Lambda$, say $[0,1) \times [0,1)$. Choose as set of marked points $P = \{ p_1 = (\frac{1}{n+1}, \frac{1}{2}), \dots, p_{n} = (\frac{n}{n+1}, \frac{1}{2}) \}$, and identify the index set $\{1 \dots n\}$ with $\bZ/n$ endowed with the natural cyclic order.\footnote{In order to make sense of the successor operator $\bullet + 1$ on the index set, I will also use the additive structure of $\bZ/n$.} A cyclic order allows us to speak unambiguously about ordered triples. If $i, j, k \in \{1 \dots n \}$ (not necessarily distinct) form an ordered triple, I shall write $i \preccurlyeq j \preccurlyeq k$. If I require $i,j,k$ to be distinct, I will use the symbol $\prec$.

Let $\alpha$ and $\beta_{i}$, $i \in \{1 \dots n \}$ be the following simple closed curves: in the fundamental domain, $\alpha$ is given by $[0,1) \times \{\frac{1}{3}\}$, and $\beta_i$ is given by $\{\frac{i}{n+1} -  \frac{1}{2(n+1)}\} \times [0,1)$. It will be important to consider also separating curves $\gamma_{i,j}$ indexed by an ordered pair $i, j \in \{1 \dots n \}$. The loop $\gamma_{i,j}$ can be described as the boundary of a tubular neighborhood of a straight segment $\sigma$ in $\bT$, starting at $p_i$ and ending at $p_j$, and such that $p_k \in \sigma$ if and only if $i \preccurlyeq k \preccurlyeq j$. A schematic representation of these curves is given in Figure \ref{fig 1}.

\begin{figure}
\includegraphics[height=2in]{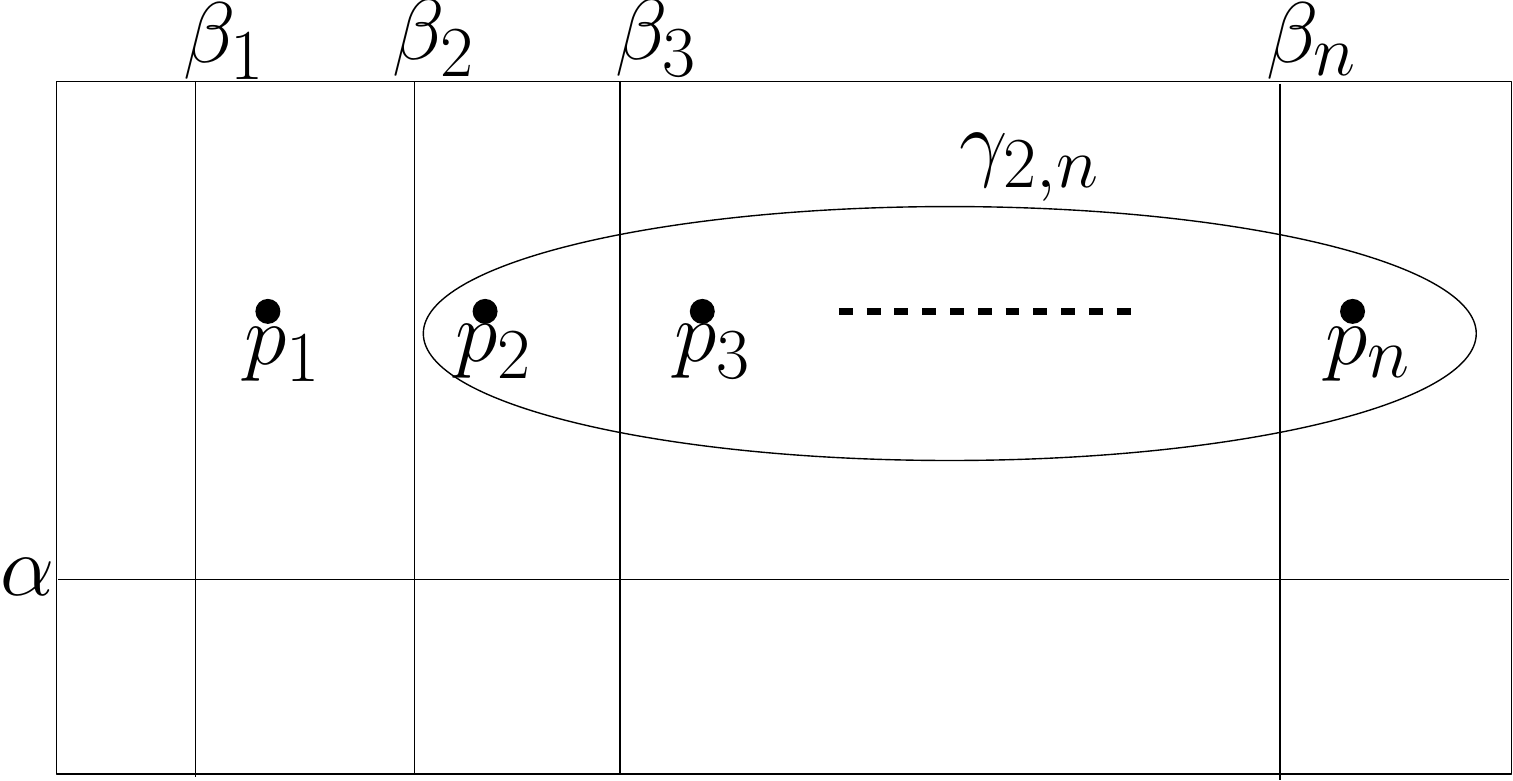}
\caption{The picture above represents the simple closed curves introduced earlier, which are visualized as subsets of the fixed fundamental domain for the action of $\Lambda$.}
\label{fig 1}
\end{figure}

If $\mu$ is a simple closed curve in a differentiable surface $\Sigma$, denote $T_\mu$ the Dehn twist along it. I will consider $\bT_{n-1,1}$ to be the closed subsurface of $\bT_n$ obtained by cutting out a small open disc centered in $p_n$ (small means that its boundary should not intersect any of the loops described above). It follows from \cite{LP} that both $\mathrm{PM}(\bT_{n})$ and $\mathrm{PM}(\bT_{n-1, 1})$ are generated by Dehn twists $T_\alpha$, and $T_{\beta_{i}}$, $i \in \{1 \dots n\}$. I will refer to this collection of Dehn twists as \emph{Humphrey generators}, in analogy with Humphrey's set of generators for the mapping class group of a compact surface. 

A presentation of $\mathrm{PM}(\bT_{n-1, 1})$ in terms of Humphrey generators can be read off Proposition 3.3 of \cite{LP}. For the reader's convenience I collect it below.
\begin{proposition}
\label{bd}
The pure mapping class group $\mathrm{PM}(\bT_{n-1, 1})$ is generated by $T_{\alpha}$, and $T_{\beta_{i}}$, $i \in \{1 \dots n \}$, subject to the following relations:
\begin{itemize}
\item(Braid relations) for every $i, j \in \{1 \dots n \}$, 

$T_{\beta_{i}}T_{\beta_{j}} = T_{\beta_{j}}T_{\beta_{i}}$,

$T_{\alpha}T_{\beta_i}T_{\alpha} = T_{\beta_i}T_{\alpha}T_{\beta_i}.$

\item(Commutativity relations)
for every $i, j, k \in \{1 \dots n \}$, $i \prec j \prec k$,
$$
\Tbi(\Ta^{-1}\Tbkk^{-1}\Tbj^{-1}\Ta^{-1}\Tbk\Ta\Tbj\Tbkk\Ta) = (\Ta^{-1}\Tbkk^{-1}\Tbj^{-1}\Ta^{-1}\Tbk\Ta\Tbj\Tbkk\Ta)\Tbi.
$$ 
\end{itemize}
\end{proposition}

An analogous presentation for $\mathrm{PM}(\bT_n)$ is described by the following Proposition.
\begin{proposition}
\label{punct}
Let $i, j \in \{1 \dots n \}$, and set 
$$
A_{i,j} = \Tbjj\Ta\Tbii^{-1}\Tbi\Ta^{-1}\Tbjj^{-1}\Ta\Tbi^{-1}\Tbii\Ta^{-1}\Tbii^{-1}\Tbi.
$$
The pure mapping class group $\mathrm{PM}(\bT_n)$ is generated by $T_{\alpha}$, and $T_{\beta_{i}}$, $i \in \{1 \dots n \}$, subject to the following relations:
\begin{itemize}
\item Braid relations and Commutativity relations (see Proposition \ref{bd}).
\item ($G$-relation) $(\Ta T_{\beta_{1}})^{6} = A_{1,n}A_{2,n} \dots A_{n-1,n}$.
\end{itemize}
\end{proposition}

Before giving a proof of Proposition \ref{punct}, it is useful to consider an alternative presentation of $\mathrm{PM}(\bT_2)$, which will play a role in the next Section, and is contained in Corollary \ref{two} below. Although Corollary \ref{two} follows quite easily from Proposition \ref{punct}, rather than discussing the details of this derivation, I refer the reader to \cite{PS} for a direct proof.

\begin{corollary}
\label{two}
The pure mapping class group $\mathrm{PM}(\bT_2)$ is generated by $\Ta$, $T_{\beta_1}$ and $T_{\beta_2}$, subject to the following relations:
\begin{itemize}
\item Braid relations
\item ($G_2$-relation) $(T_{\beta_1} \Ta T_{\beta_2})^4 = 1$.
\end{itemize}
\end{corollary}

The proof of Proposition \ref{punct} depends on Proposition \ref{bd}, and Lemma \ref{punct-bd}. It follows immediately from the definition of Dehn twist that, for all $i \in \{1 \dots n\}$, $T_{\gamma_{i,i}} = 1$ in $\mathrm{PM}(\bT_{n})$. In $\bT_{n-1, 1}$, however, $\gamma_{n,n}$ is isotopic to the boundary component, and therefore $T_{\gamma_{n,n}}$ defines a non trivial mapping class. Lemma \ref{punct-bd} assures us that the only extra relation needed to obtain $\mathrm{PM}(\bT_{n})$ from the presentation in Proposition \ref{bd} is precisely $T_{\gamma_{n,n}} = 1$. What is left to do is finding an expression for $T_{\gamma_{n,n}}$ as a product of Humphrey generators. To achieve this, I need to introduce two more ingredients. The first is the following lemma,

\begin{figure}
\includegraphics[height=1.8in]{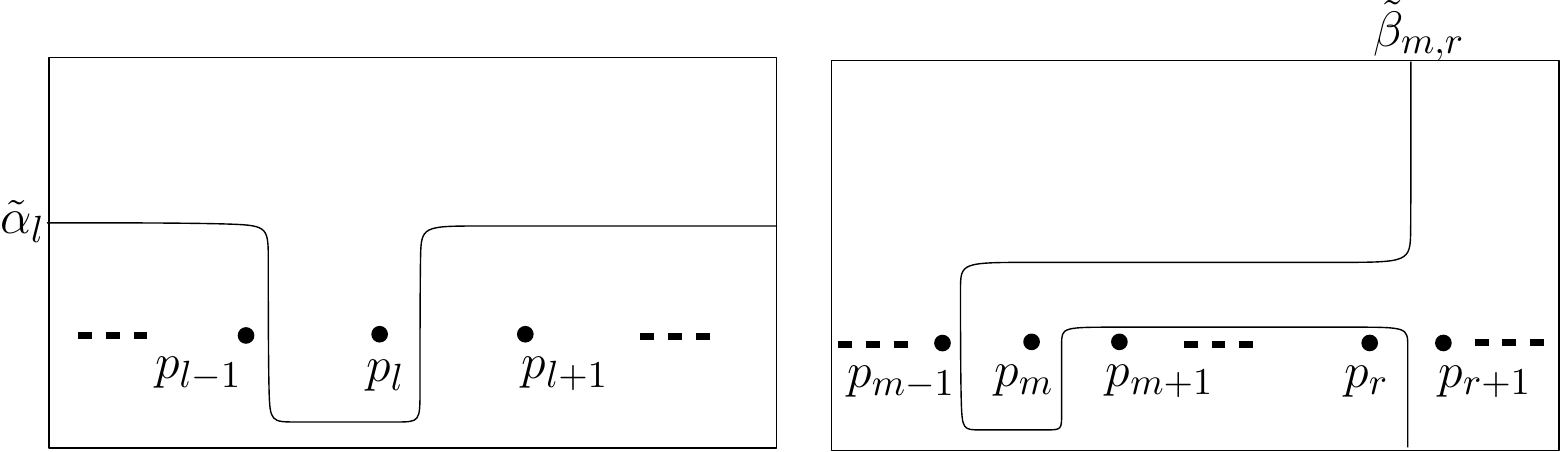}
\caption{Above is a picture of the simple closed curves $\tilde{\alpha}_m$, and $\tilde{\beta}_{n,r}$, which are discussed in the proof of Lemma \ref{recursion}.}
\label{fig 2}
\end{figure}

\begin{lemma}
\label{recursion}
Let $i, j \in \{1 \dots n \}$, and, as in Proposition \ref{punct}, set  
$$
A_{i,j} = \Tbjj\Ta\Tbii^{-1}\Tbi\Ta^{-1}\Tbjj^{-1}\Ta\Tbi^{-1}\Tbii\Ta^{-1}\Tbii^{-1}\Tbi,
$$
then $T_{\gamma_{i,j}} = A_{i,j}T_{\gamma_{i+1,j}}.$
\end{lemma} 
\begin{proof}
Let $\bT_{n-1}$ be the torus with $n-1$ punctures obtained from $\bT_n$ by filling in the puncture $p_i$. The Birman exact sequence (see \cite{FM}, Theorem 4.6), applied to the inclusion $\bT_n \hookrightarrow \bT_{n-1}$, yields
$$
1 \rightarrow \pi_1(\bT_{n-1}, p_i) \stackrel{\mathcal{P}ush}{\longrightarrow} \mathrm{PM}(\bT_{n}) \stackrel{\mathcal{F}orget}{\longrightarrow}  \mathrm{PM}(\bT_{n-1}) \rightarrow 1,
$$
where $\pi_1(\bT_{n-1}, p_i)$ is the fundamental group of $\bT_{n-1}$ with base-point $p_i$. The names attached to the maps above follow the conventions of Chapter 4 in \cite{FM}, to which I refer the reader for further details on the Birman exact sequence.

The key point is that $T_{\gamma_{i,j}}T_{\gamma_{i+1,j}}^{-1}$ lies in the image of the morphism $\mathcal{P}ush$. Figure \ref{fig 2} describes the geometry of two classes of simple closed curves in $\bT_n$, called respectively $\tilde{\alpha}_m$, and $\tilde{\beta}_{n,r}$, $m,n,r \in \{1 \dots n \}$. It immediately follows from the definition of $\mathcal{P}ush$ that, in $\mathrm{PM}(\bT_{n})$, 
$$
T_{\gamma_{i,j}}T_{\gamma_{i+1,j}}^{-1} = T_{\beta_i}T_{\beta_{i+1}}^{-1}T_{\beta_j}T_{\tilde{\beta}_{i,j}}^{-1}.
$$

It is not hard to express $T_{\tilde{\beta}_{i,j}}$ in terms of Humphrey generators. In fact, by simply applying the definition of Dehn twist, one can verify that $\tilde{\beta}_{i,j} = T_{\tilde{\alpha_i}}T_{\alpha}^{-1}(\beta_j)$, and $\tilde{\alpha_i} = T_{\beta_i}^{-1}T_{\beta_{i+1}}(\alpha)$.\footnote{Note that here, as everywhere in the paper, I am considering curves only up to isotopy.}
Now recall that, if $\mu$ and $\mu'$ are two simple closed curves in an oriented surface $\Sigma$, then $T_{T_{\mu}(\mu')} = T_{\mu}T_{\mu'}T_{\mu}^{-1}$ (this is Fact 3.7 in \cite{FM}). Thus 
$$
T_{\tilde{\beta}_{i,j}} = T_{\tilde{\alpha_i}}T_{\alpha}^{-1}T_{\beta_j}T_{\alpha}T_{\tilde{\alpha_i}}^{-1}, \text{ and } T_{\tilde{\alpha_i}} = T_{\beta_i}^{-1}T_{\beta_{i+1}} T_{\alpha} T_{\beta_{i+1}}^{-1}T_{\beta_i}.
$$
Using this last identity, we can rewrite first $T_{\tilde{\beta}_{i,j}}$, and then $T_{\gamma_{i,j}}T_{\gamma_{i+1,j}}^{-1}$, as a product of Humphrey generators, and this completes the proof of Lemma \ref{recursion}.
\end{proof}

The second ingredient is given by a family of relations in the mapping class group, introduced by Gervais in \cite{G} as \emph{star relations}.
\begin{proposition}
\label{stars}
Let $i, j, k \in \{1 \dots n \}$, and $i \preccurlyeq j \preccurlyeq k$. Then 
$$
(\Ta\Tbi\Tbj\Tbk)^{3} = T_{\gamma_{i,j}}T_{\gamma_{j,k}}T_{\gamma_{k,i}}.
$$
\end{proposition} 
\begin{proof}
See Theorem 1 in \cite{G}.
\end{proof}

Note that, when $i = j = k$, one obtains the following `degenerate' star relations,
$$
(\Ta\Tbi\Tbi\Tbi)^3 = T_{\gamma_{i, i-1}}.
$$ 
Using the braid relations, the product on the LHS of the equality can be rewritten as $(\Ta\Tbi)^6$, and therefore Proposition \ref{stars} yields, for all $i \in \{1, \dots, n\}$, the identity
$$
(\Ta\Tbi)^6 = T_{\gamma_{i, i-1}}.
$$

Let us fix $i \in \{1 \dots n\}$, say $i = 1$. Then the degenerate star identity for $i=1$ combined with an iterated application of Lemma \ref{recursion} (from which we import the notation), gives the formula
$$
(\Ta T_{\beta_{1}})^{6} = (A_{1,n}A_{2,n} \dots A_{n-1,n})T_{\gamma_{n,n}}.
$$   
Since the $A_{i,j}$-s are defined as a product of $\Ta$ and $\Tbi$-s, this yields the sought after expression of $T_{\gamma_{n,n}}$ in terms of Humphrey generators, and concludes the proof of Proposition \ref{punct}. \\

Lemma \ref{G} below is the last result of this Section, and describes a family of identities in $\mathrm{PM}(\bT_{n-1, 1})$, which will be useful in Section \ref{action}. 
\begin{lemma}
\label{G}
If $i \in \{ 1, \dots, n \}$, then 
$$
(\Ta T_{\beta_{1}})^{6}(A_{1,n}A_{2,n} \dots A_{n-1,n})^{-1} = (\Ta T_{\beta_{i}})^6(A_{i, i+n-1}A_{i+1, i+n-1} \dots A_{i + n-2,i+n-1})^{-1}
$$ 
as elements of $\mathrm{PM}(\bT_{n-1, 1})$.
\end{lemma}

Before proceeding with the proof of Lemma \ref{G}, a few comments are in order. Note that the $G$-relation of Proposition \ref{punct} depends on the degenerate star identity for $i=1$. However, because of the evident cyclic symmetry of the problem, in $\mathrm{PM}(\bT_n)$ one would have more generally, for any $i \in \{1 \dots n\}$, the identity 
$$
(\Ta T_{\beta_{i}})^{6} = A_{i, i+n-1}A_{i+1, i+n-1} \dots A_{i + n-2,i+n-1}.
$$
As a consequence, the following chain of equalities holds in $\mathrm{PM}(\bT_n)$, 
$$
(\Ta T_{\beta_{1}})^{6}(A_{1,n}A_{2,n} \dots A_{n-1,n})^{-1} = (\Ta T_{\beta_{i}})^6(A_{i, i+n-1}A_{i+1, i+n-1} \dots A_{i + n-2,i+n-1})^{-1} = 1.
$$
Lemma \ref{G} asserts that, in fact, the first of these two equalities can be lifted to $\mathrm{PM}(\bT_{n-1, 1})$.

\begin{proof}[Proof of Lemma \ref{G}]
Consider the element $G' \in \mathrm{PM}(\bT_{n-1, 1})$ obtained by multiplying the expression on the LHS of the equality, by the inverse of the expression on the RHS, that is
$$
G' = (\Ta T_{\beta_{i}})^6(A_{i, i+n-1}A_{i+1, i+n-1} \dots A_{i + n-2,i+n-1})^{-1}((\Ta T_{\beta_{1}})^{6}(A_{1,n}A_{2,n} \dots A_{n-1,n})^{-1})^{-1}.
$$ 
Also, set $G = (\Ta T_{\beta_{i}})^{6}(A_{i, i+n-1}A_{i+1, i+n-1} \dots A_{i + n-2,i+n-1})^{-1}$. 
As I pointed out above, the image of $G'$ in $\mathrm{PM}(\bT_n) = \mathrm{PM}(\bT_{n-1, 1})/\langle G \rangle$ is equal to $1$. Since $G$ is central it follows that $G'$ must be a power of $G$, that is, in $\mathrm{PM}(\bT_{n-1, 1})$ $G' = G^n$ for some $n \in \bZ$. I will show that $n = 0$. This implies that $G' = 1$ in $\mathrm{PM}(\bT_{n-1, 1})$, and proves Lemma \ref{G}.

The identity $G' = G^n$ is equivalent to the following,
\begin{equation}
(\Ta T_{\beta_{i}})^6(A_{i, i+n-1}A_{i+1, i+n-1} \dots A_{i + n-2,i+n-1})^{-1} = ((\Ta T_{\beta_{1}})^{6}(A_{1,n}A_{2,n} \dots A_{n-1,n})^{-1})^{n+1}.
\label{eq G' G}
\end{equation}
Recall that there is a homomorphism $\mathrm{PM}(\bT_{n-1,1}) \stackrel{\mathcal{F}orget}{\longrightarrow} \mathrm{PM}(\bT_{0,1})$,\footnote{By $\bT_{0,1}$ I mean a symplectic torus with no punctures, and one boundary component.}  which generalizes the map of the same name appearing in Birman exact sequence (see \cite{FM}, Section 9.1 for more details). Since the map $\mathcal{F}orget$ is induced by the inclusion $\bT_{n-1,1} \hookrightarrow \bT_{0,1}$, and all the $\beta_i$-s have identical isotopy class as subsets of $\bT_{0,1}$, we have that for all $i,j \in \{1, \dots, n\}$,
$$
\mathcal{F}orget(\Tbi) = \mathcal{F}orget(\Tbj) =: T_{\beta}, \text{ and } \mathcal{F}orget(A_{i,j}) = 1. 
$$
Applying $\mathcal{F}orget$ to both sides of equation (\ref{eq G' G}), yields therefore the identity
$$
(\Ta T_{\beta})^6 = (\Ta T_{\beta})^{6(n+1)}
$$
in $\mathrm{PM}(\bT_{0,1})$. As explained by Corollary 7.3 in \cite{FM}, there are no torsion elements in the mapping class group of a surface $\Sigma$, provided that its boundary set is non-empty. This is indeed the case of $\bT_{0,1}$, and thus $(n+1)$ must equal $1$, as desired.
\end{proof}

\section{The action of $\widetilde{\mathrm{PM}}(\bT_n)$ on $D^bCoh(X_n)$.} 
\label{action}
Let $X_n$ be a cycle of $n$ projective lines over a field $\gk$. That is, $X_n$ is a connected reduced curve with $n$ nodal singularities, such that its normalization $\tilde{X_n} \stackrel{\pi}{\rightarrow} X$ is a disjoint union of $n$ projective lines $D_1, \dots, D_n$, with the property that the pre-image along $\pi$ of the singular set interesects each $D_i$ in exactly two points. Following the discussion in Section 1 of \cite{ST}, the group acting on $D^bCoh(X_n)$ is going to be a suitable central extension of $\mathrm{PM}(\bT_n)$, whose elements should be viewed as \emph{graded} symplectic automorphisms of the mirror of $X_n$, i.e. the torus with $n$ punctures.

\begin{definition}
\label{ext}
Define $\widetilde{\mathrm{PM}}(\bT_n)$ as the $\bZ$-central extension of $\mathrm{PM}(\bT_n)$,
$$
0 \rightarrow \bZ \rightarrow \widetilde{\mathrm{PM}}(\bT_n) \rightarrow \mathrm{PM}(\bT_n) \rightarrow 1,
$$
generated by $\Ta$, $\Tbi$ $i \in \{1 \dots n \}$, and a central element $t$, subject to the following relations:
\begin{itemize}
\item Braid relations and Commutativity relations, as in Proposition \ref{punct},
\item ($\tilde{G}$-relation) $(\Ta T_{\beta_{1}})^{6}(A_{1,n}A_{2,n} \dots A_{n-1,n})^{-1} = t^2$.
\end{itemize}
\end{definition} 

\begin{remark}
\label{twotwo}
By lifting the $G_2$-relation of Corollary \ref{two} to the central extension, one can give an alternative presentation of $\widetilde{\mathrm{PM}}(\bT_2)$ in which the $\tilde{G}$-relation of Definition \ref{ext} is replaced by the following, 
\begin{itemize}
\item ($\tilde{G}_2$-relation) $(T_{\beta_1} \Ta T_{\beta_2})^4 = t^2$.
\end{itemize}
\end{remark}
The theory of spherical objects was introduced by Seidel and Thomas in \cite{ST}. Given a triangulated category $C$, under mild assumptions, to any object $\mathcal{E}$ in $C$ such that $Hom^{*}(\mathcal{E},\mathcal{E})$ is isomorphic to the cohomology of the $n$-sphere (i.e. a \emph{spherical object}), one can associate an autoequivalence, called \emph{twist}, $T_{\mathcal{E}}: C \rightarrow C$. 

Let $x_1 \dots x_n \in X_n$ be (closed) smooth points, such that $x_i$ lies on the $i$-th irreducible component of $X_n$. It is easy to see that the sheaves $\cO = \cO_{X_n}$, $\gk(x_i)$ $i \in \{1 \dots n \}$ in $D^b(Coh(X_n))$ are 1-spherical, and therefore determine twist functors $T_{\cO},  T_{\gk(x_i)}$. These equivalences, together with the \emph{shift} functor, will define the action of $\widetilde{\mathrm{PM}}(\bT_n)$ on $D^b(Coh(X_n))$. 
The main reference for the computations below are \cite{ST} and \cite{BK}. In \cite{BK} the reader can find a detailed treatment of the case $n =1$, while in \cite{ST} Seidel and Thomas discuss the smooth case, i.e. the action of the mapping class group of a torus with no marked points on the derived category of a smooth elliptic curve.  

The following lemma will be extremely useful for computations. 
\begin{lemma}
\label{equiv}
Let $F: D^b(Coh(X_n)) \rightarrow D^b(Coh(X_n))$ be an auto-equivalence of triangulated categories. If 
\begin{itemize}
\item $F(\cO) \cong \cO$, and 
\item for all $i \in \{1 \dots n \}$, $F(\gk(x_i)) \cong \gk(x_i)$, 
\end{itemize}
then there exists an isomorphisms $f: X_n \rightarrow X_n$, such that $F$ is naturally equivalent to $f^*: D^b(Coh(X_n)) \rightarrow D^b(Coh(X_n))$.
\end{lemma}
\begin{proof}
Note that $X_n$ is projective, as $X_1$ is isomorphic to a nodal cubic curve in $\bP^2$, $X_2$ can be embedded as the union of a line and a quadric in $\bP^2$, and, if $n \geq 3$, $X_n$ can be embedded as a union of $n$ linear subspaces in $\bP^{n-1}$. Consider the line bundle $\cL = \cO(x_1 + \dots + x_n)$ over $X_n$. $\cL$ is ample (and very ample for $n \geq 3$). Since $F$ preserves $\cO$ and $\gk(x_i)$, it is easy to see that $F(\cL^{\otimes m}) \cong \cL^{\otimes m}$  for all $m \in \bZ$. In fact, $\cL^{-1}$ is isomorphic to the kernel of any surjective morphism of sheaves $p: \cO \rightarrow \bigoplus_{i=1}^{i=n} \gk(x_i)$. $F(\cL^{-1})$ is therefore isomorphic to the (co-)cone of the map 
$$
F(\cO) (\cong \cO) \stackrel{F(p)}{\rightarrow} F(\bigoplus_{i=1}^{i=n} \gk(x_i)) (\cong \bigoplus_{i=1}^{i=n} \gk(x_i)),
$$
where $F(p)$ must be surjective. It follows that $F(\cL^{-1}) \cong \cL^{-1}$. Similarly $\cL$ is isomorphic to the cone of any morphism in $Hom^1(\bigoplus_{i=1}^{i=n} \gk(x_i) , \cO )$ corresponding, under Serre duality, to a surjective morphism $p$ as above, and thus $F(\cL) \cong \cL$. Analogous arguments can be made for all the tensor powers of $\cL$. 

From here, in order to prove the claim, is sufficient to mimic the proof of Theorem 3.1 of \cite{BO} (see also Proposition 6.18 in \cite{Ba}, in which the argument from \cite{BO} is applied, as here, in the context of singular algebraic varieties). A brief summary of the argument goes as follows. Note first that the functor $F$ induces a graded automorphism of the homogeneous coordinate algebra $\bigoplus_{m=0}^{m=\infty}H^0(\cL^{\otimes m})$, which, up to rescaling, must be given by the pull-back along an automorphism $f: X_n \rightarrow X_n$. Call $C$ the full linear sub-category of $D^b(Coh(X_n))$ having as objects $\{ \cL^{\otimes m} \}_{m \in \bZ}$. As explained in \cite{BO}, one can define a natural equivalence between the restrictions to $C$ of $F$ and $f^*$. Further, since $X_n$ is projective, and $\cL$ is ample, $\{\cL^{\otimes m} \}_{m \in \bZ}$ form an ample sequence in the sense of \cite{BO} (for a proof of this, see Proposition 3.18 of \cite{Huy}). The claim then follows from Proposition A.3 of \cite{BO}, which implies that the natural equivalence $F \cong f^*$ over $C$ can be extended to the full derived category $D^b(Coh(X_n))$.
\end{proof}

\begin{remark}
\label{map}
Note that if $f: X_n \rightarrow X_n$ is an automorphism such that $f(x_i) = x_i$, and $n \geq 3$, then $f$ is the identity. If $n \leq 2$, $f$ may be non-trivial and act as a (non-trivial) permutation on the pre-image of the singular locus in the normalization. However it is immediate to see that $f$ is an involution, i.e. $f^2 = Id$. 
\end{remark}

\begin{lemma}
\label{comp}
Let $x \in X_n$ be a smooth point, then
\begin{itemize}
\item $T_{\gk(x)} \cong -\otimes \cO(x)$,
\item $\To(\gk(x)) \cong \cO(-x)[1]$,
\item $\To(\cO(x)) \cong \gk(x)$,
\item $\To(\cO) \cong \cO$.
\end{itemize}
\end{lemma}
\begin{proof}
The first isomorphism is proved in \cite{ST}, Section 3.d. For the other isomorphisms, see Lemma 2.13 in \cite{BK}.
\end{proof}

I am now ready to state the main theorem of this paper.
\begin{theorem}
\label{main}
The assignment
\begin{itemize}
\item for all $i \in \{1 \dots n\}$, $\Tbi \mapsto \Txi$,
\item $\Ta \mapsto \To$, and
\item $t \mapsto [1]$,
\end{itemize}
defines a weak action of $\widetilde{\mathrm{PM}}(\bT_n)$ on $D^b(Coh(X_n)).$
\end{theorem}
Following \cite{ST}, by weak action I mean that this assignment defines a homorphism between $\widetilde{\mathrm{PM}}(\bT_n)$, and the group of autoequivalences of $D^b(Coh(X_n))$ modulo natural isomorphism of functors. The action defined in Theorem \ref{main} depends on the choice of $x_1, \dots, x_n$. However, the action is unique up to conjugation. Note that there is a natural $(\bC^*)^n$-action on $X_n$, with the property that the $i$-th copy of $\bC^*$ acts by multiplication on the $i$-th component of $X_n$. Let $\lambda = (\lambda_1, \dots, \lambda_n) \in (\bC^*)^n$, and let $m_{\lambda}: X_n \rightarrow X_n$ be the associated automorphism. Then one can show that, for all $i \in \{1, \dots, n \}$, $(m_{\lambda}^*) \Txi (m_{\lambda}^*)^{-1} = T_{k(\lambda_i x_i)}$, and $(m_{\lambda}^*) \To (m_{\lambda}^*)^{-1} = \To$.

\begin{proof}[Proof of Theorem \ref{main}]
I will show that Theorem \ref{main} gives a well-defined homomorphism, by checking that the relations in Definiton \ref{ext} hold. 

\emph{Braid relations.} For all $i, j \in \{1, \dots, n\}$, $\cO$, $\gk(x_i)$ and $\gk(x_j)$ form an $A_2$-configuration, in the language of \cite{ST}. The fact that such a collection of spherical twists satisfies the braid relations is proved in Proposition 2.13 of \cite{ST}.

\emph{Commutativity relations.} By Lemma 2.11 of \cite{ST}, if $\cE_1, \cE_2$ are spherical objects, then $T_{\cE_2}T_{\cE_1}T_{\cE_2}^{-1} \cong T_{T_{\cE_2}(\cE_1)}$. It follows that, in order to prove the Commutativity relations, is sufficient to show that,
for every $i, j, k \in \{1 \dots n \}$, $i \prec j \prec k$, 
$$
\To^{-1}\Txkk^{-1}\Txj^{-1}(\To^{-1}\Txk\To)\Txj\Txkk\To(\gk(x_i)) \cong \gk(x_i) \Leftrightarrow
$$
$$
\To^{-1}\Txkk^{-1}\Txj^{-1}(\Txk\To\Txk^{-1})\Txj\Txkk\To(\gk(x_i)) \cong \gk(x_i) \Leftrightarrow \footnote{\label{foot:iso}The isomorphism $\To^{-1}\Txk\To \cong \Txk\To\Txk^{-1}$ follows immediately from the braid relations.}
$$
$$
\To\Txk^{-1}\Txj\Txkk\To(\gk(x_i)) \cong \Txk^{-1}\Txj\Txkk\To(\gk(x_i)) .
$$
By Lemma \ref{comp} $\Txk^{-1}\Txj\Txkk\To(\gk(x_i)) \cong \cO(-x_i + x_j - x_k + x_{k+1})[1]$. Thus, I need to show that
$\To(\cO(-x_i + x_j - x_k + x_{k+1})) \cong \cO(-x_i + x_j - x_k + x_{k+1})$. Proposition 2.12 of \cite{ST} states that if $\cE_1$, $\cE_2$ are spherical objects such that $Hom^i(\cE_1, \cE_2) = 0$ for all $i$, then $T_{\cE_1}(\cE_2) \cong \cE_2$. The Commutativity relations reduce therefore to the claim: for all $i, j, k \in \{1 \dots n \}$, $i \prec j \prec k$,
$$
H^0(\cO(-x_i + x_j - x_k + x_{k+1})) = H^1(\cO(-x_i + x_j - x_k + x_{k+1})) = 0.
$$
This follows from Theorem 2.2 of \cite{DGK}, which gives a general formula for computing the cohomology groups of indecomposable vector bundles over a cycle of projective lines.

\emph{$\tilde{G}$-relation.} Assume first that $n \geq 3$. I will handle separately the case $n = 2$, for which I will use the alternative $\tilde{G}_2$-relation of Remark \ref{twotwo} (for the case $n =1$, the reader should refer to \cite{BK}). Let $i, j \in \{1 \dots n \}$, and define
$$
E_{i,j} = \Txjj\To\Txii^{-1}\Txi\To^{-1}\Txjj^{-1}\To\Txi^{-1}\Txii\To^{-1}\Txii^{-1}\Txi.
$$
I need to prove that $(\To T_{k(x_{1})})^{6} \cong (E_{1,n}E_{2,n} \dots E_{n-1,n})[2]$. After Lemma \ref{equiv} and Remark \ref{map}, it is sufficient to check the $\tilde{G}$-relation on $\cO$, and $\gk(x_i)$ for all $i \in \{1 \dots n\}$. In fact, in view of Lemma \ref{G}, it is enough to evaluate the $\tilde{G}$-relation on $\cO$ and $\gk(x_1)$, as, for any $k \in \{1 \dots n \}$,
$$
(\To T_{\gk(x_{1})})^{6}(\gk(x_k)) \cong (E_{1,n}E_{2,n} \dots E_{n-1,n})[2](\gk(x_k)) \Leftrightarrow
$$
$$
(\To T_{\gk(x_{k})})^{6}(\gk(x_k)) \cong (E_{i, i+n-1}E_{i+1, i+n-1} \dots E_{i + n-2,i+n-1})[2](\gk(x_k)),
$$
and, by the cyclic symmetry of the problem, the latter identity is proved in exactly the same way as the claim that the $\tilde{G}$-relation holds for $\gk(x_1)$.

$\bullet$ $\tilde{G}$-relation on $\cO$. Simply by keeping track of the isomorphisms collected in Lemma \ref{comp}, one can see that
$$
(\To\Txu)^6(\cO) \cong (\To\Txu)^4(\cO(-x_1)[1]) \cong (\To\Txu)^2(\gk(x_1)[1]) \cong \cO[2].
$$
On the other hand, I will show that, for all $i,j \in \{1, \dots, n\}$, $E_{i,j}(\cO) \cong \cO$, and therefore 
$$
(E_{1,n}E_{2,n} \dots E_{n-1,n})[2](\cO) \cong \cO[2],
$$
as expected. Note first that if $x, y \in X_n$ are smooth points lying on different connected components, then $\To(\cO(x - y)) \cong \cO(x - y)$. This again follows from Theorem 2.2 of \cite{DGK}, but can also be checked directly using the braid relations. Using this isomorphism, and Lemma \ref{comp}, it is easy to check that
$$
E_{i,j}(\cO) = \Txjj\To\Txii^{-1}\Txi\To^{-1}\Txjj^{-1}\To\Txi^{-1}\Txii\To^{-1}\Txii^{-1}\Txi(\cO) \cong 
$$
$$
\Txjj\To\Txii^{-1}\Txi\To^{-1}\Txjj^{-1}\To\Txi^{-1}\Txii(\cO(x_i - x_{i+1})) \cong 
$$
$$
\Txjj\To\Txii^{-1}\Txi\To^{-1}\Txjj^{-1}(\cO) \cong 
$$
$$
\Txjj\To\Txii^{-1}(\gk(x_{j+1})[-1]) \cong \cO.
$$

$\bullet$ $\tilde{G}$-relation on $\gk(x_1)$. As before, it is enough to apply Lemma \ref{comp} to see that
$$
(\To\Txu)^6(\gk(x_1)) \cong (\To\Txu)^4(\cO[1]) \cong (\To\Txu)^2(\cO(-x_1)[2]) \cong \gk(x_1)[2].
$$
Further, for all $i \in \{1, \dots, n-1\}$, $E_{i,n}(\gk(x_1)) \cong \gk(x_1)$. In fact, 
$$
E_{i,n}(\gk(x_1)) = \Txu\To\Txii^{-1}\Txi\To^{-1}\Txu^{-1}\To\Txi^{-1}\Txii\To^{-1}\Txii^{-1}\Txi(\gk(x_1)) \cong
$$
$$
\Txu\To\Txii^{-1}\Txi\To^{-1}\Txu^{-1}\To\Txi^{-1}\Txii(\cO(x_1)).
$$
Now, 
$$
\Txu\To\Txii^{-1}\Txi\To^{-1}\Txu^{-1}\To\Txi^{-1}\Txii(\cO(x_1)) \cong \gk(x_1) \Leftrightarrow
$$
$$
\Txii^{-1}\Txi(\To^{-1}\Txu^{-1}\To)\Txi^{-1}\Txii(\cO(x_1)) \cong \cO(x_{1}) \Leftrightarrow \footnote{The isomorphism $\To^{-1}\Txu^{-1}\To \cong \Txu\To^{-1}\Txu^{-1}$ is obtained from the one in Footnote \ref{foot:iso}, by taking inverses on both sides of the ``$\cong$'' sign.}
$$
$$
\Txii^{-1}\Txi(\Txu\To^{-1}\Txu^{-1})\Txi^{-1}\Txii(\cO(x_1)) \cong \cO(x_{1}) \Leftrightarrow
$$
$$
\To^{-1}\Txu^{-1}\Txi^{-1}\Txii(\cO(x_1)) \cong \cO(x_{i+1} - x_i) \Leftrightarrow
$$
$$
\cO(x_{i+1} - x_i) \cong \To(\cO(x_{i+1} - x_i)). 
$$
As I pointed out above, this last isomorphism can be proved using the braid relations. Thus
$$
(E_{1,n}E_{2,n} \dots E_{n-1,n})[2](\gk(x_1)) \cong \gk(x_1)[2],
$$
and this concludes the proof of Theorem \ref{main} for the case $n \geq 3$.

\emph{The case n = 2}. Note that there are isomorphisms
\begin{itemize}
\item $(T_{\gk(x_1)}T_{\cO}T_{\gk(x_2)})^2(\cO) \cong \cO[1]$, and
\item $(T_{\gk(x_1)}T_{\cO}T_{\gk(x_2)})^2(\gk(x_1)) \cong \gk(x_{2})[1]$, $(T_{\gk(x_1)}T_{\cO}T_{\gk(x_2)})^2(\gk(x_2)) \cong \gk(x_{1})[1]$. 
\end{itemize}
Let us check this for $\gk(x_1)$:
$$
(T_{\gk(x_1)}T_{\cO}T_{\gk(x_2)})(T_{\gk(x_1)}T_{\cO}T_{\gk(x_2)})(\gk(x_1)) \cong (T_{\gk(x_1)}T_{\cO}T_{\gk(x_2)})(\cO[1]) \cong \gk(x_2)[1].
$$
Consider an involution $\sigma: X_2 \rightarrow X_2$ such that $\sigma(x_1) = x_2$, and $\sigma(x_2) = x_1$. It follows from Remark \ref{map} that there is an isomorphism $f: X_2 \rightarrow X_2$, and a natural equivalence $(T_{\gk(x_1)}T_{\cO}T_{\gk(x_2)})^2 \cong f^* \sigma^* [1]$. As $\sigma$ and $f$ commute, by taking the square of this natural equivalence, one gets
$$
(T_{\gk(x_1)}T_{\cO}T_{\gk(x_2)})^4 \cong (f^* \sigma^* [1]) (f^* \sigma^* [1]) \cong (f^* )^2 (\sigma^*)^2 [2] \cong [2].
$$
In view of Remark \ref{twotwo}, this implies that the action of $\widetilde{\mathrm{PM}}(\bT_2)$ on $D^b(Coh(X_2))$ is well defined, and proves the case $n = 2$ of Theorem \ref{main}.
\end{proof}

\begin{remark}
\label{symplectic}
It follows from results in Appendix D of \cite{B}, that the action defined by Theorem \ref{main} is, in an appropriate sense, a categorification of the symplectic representation of the mapping class group. Denote $K^{num}(X_n)$ the quotient of $K_0(\Perf(X_n))$ by the radical of the Euler form (see Appendix D of \cite{B} for further details). The Euler form induces a non-degenerate skew-symmetric form on $K^{num}(X_n)$, and there is an isomorphism of symplectic lattices $K^{num}(X_n) \cong H_1(\bT_n, \bZ) (\cong \bZ^{n+1})$ (here, $\bT_n$ denotes the torus with the $n$ marked points removed). Note that the induced action of $\tilde{\mathrm{PM}}(\bT_n)$ on $K^{num}(X_n)$ factors through $\mathrm{PM}(\bT_n)\oplus \bZ_2$. Bodnarchuk's computations imply that the resulting action of $\mathrm{PM}(\bT_n)$ on $K^{num}(X_n)$ is isomorphic to the standard symplectic representation of $\mathrm{PM}(\bT_n)$ over $H_1(\bT_n, \bZ)$.
\end{remark}

\end{document}